\documentclass{amsart}

\usepackage{amsfonts}
\usepackage{amsmath}
\usepackage{amsthm}
\usepackage{amssymb}
\usepackage{latexsym}
\usepackage{enumerate}
\usepackage[scr=zapfc]{mathalpha}

\newtheorem{theorem}{Theorem}[section]
\newtheorem{lemma}[theorem]{Lemma}
\newtheorem{proposition}[theorem]{Proposition}

\newtheorem{question}[theorem]{Question}
\theoremstyle{definition}
\newtheorem{definition}[theorem]{Definition}

\numberwithin{equation}{section}

\newcommand{\N}{\mathbb{N}}

\newcommand{\Q}{\mathbb{Q}}
\newcommand{\R}{\mathbb{R}}

\newcommand{\SSS}{\mathbb{S}}
\newcommand{\forces}{\Vdash}

\newcommand{\I}{\mathcal{I}}

\newcommand{\X}{\mathcal{X}}

\newcommand{\coded}[1]{\mathscr{#1}}
\newcommand{\cX}{\coded{X}}
\newcommand{\cY}{\coded{Y}}
\newcommand{\cZ}{\coded{Z}}

\DeclareMathOperator{\dist}{dist}
\DeclareMathOperator{\supp}{supp}
\DeclareMathOperator{\diam}{diam}
\DeclareMathOperator{\lev}{lev}

\DeclareMathOperator{\cf}{cf}

\title[]{On disjoint tilings of $\ell_1(\kappa)$ by star-shaped bodies}

\author[P. Koszmider]{Piotr Koszmider}
\address{Institute of Mathematics of the Polish Academy of Sciences, \'Sniadeckich 8,
00-656 Warsaw, Poland}
\email{piotr.math@proton.me}
\urladdr{https://piotrkoszmider.github.io}

\author[P. Szewczak]{Piotr Szewczak}
\address{Institute of Mathematics, University of Warsaw, Banacha 2,
02-097 Warsaw, Poland}
\email{p.szewczak@wp.pl}
\urladdr{http://piotrszewczak.pl}

\thanks{The  first-named author was  partially supported by the NCN (National Science
Center, Poland) research grant no.\ 2020/37/B/ST1/02613.}



\begin{document}

\begin{abstract}
We employ the set-theoretic methods
yielding the consistency of the existence of nontrivial pairwise disjoint covers
 of the unit interval (or equivalently $\R^n$ for $n\in\N\setminus\{0\}$) by less than $2^\omega$ closed sets 
in the context of  covers of infinite dimensional Banach spaces $\ell_1(\kappa)$
for infinite $\kappa$ under geometric conditions arising 
in  tiling theory and in  approximation theory.

Specifically, for $\kappa=\omega_1, \omega_2$, using side-by-side Sacks forcing we prove the consistency 
of an arbitrarily large continuum above
$\kappa$ with the existence
 of a highly disconnected proximinal set in $\ell_1(\kappa)$, built from  norm-compact
 pieces separated by a common positive distance, and the existence of 
a normal disjoint tiling of $\ell_1(\kappa)$ by star-shaped bodies of the form $X+B$,
where $X$ is norm compact and $B$ is the unit ball.
We also prove that if $\kappa$ is an uncountable cardinal
of countable cofinality, then $\ell_1(\kappa)$ does not admit
any normal disjoint tiling by bodies of the form $X+B$, where
$X$ is closed and norm-separable and $B$ is the unit ball.

These results complement a result of Klee of 1981 obtained for
$\kappa$ satisfying $\kappa^\omega=\kappa$, recent results of 
De Bernardi, Russo, Sezgek and Somaglia, and a  {\sf ZFC}
a machine-discovered result  (included in the appendix) that there are no nontrivial discrete Chebyshev sets 
in $\ell_1(\kappa)$ when $\kappa<2^\omega$. 
\end{abstract}

\maketitle

\section{Introduction}

By the Sierpi\'nski theorem (\cite{sierpinski})   $\R$ 
can not be covered by a nontrivial countable   pairwise disjoint family of closed subsets.
Essentially the same arguments combined with the classical consequences 
of  Martin's axiom prove that it is consistent that the continuum hypothesis fails and $\R$ 
can not be covered by a nontrivial  pairwise disjoint family of closed sets of cardinality
less than continuum.
However, Stern has  proved in \cite{stern} that  it is consistent that the continuum hypothesis fails and $\R$ 
can  be covered by $\omega_1$ many pairwise disjoint closed subsets.
In fact it is proved in Theorem 3 of \cite{miller} that the
existence of a disjoint cover of cardinality $\omega_1$ by closed sets is equivalent
for any two uncountable Polish spaces (cf. Lemma 2.2 of \cite{brian}).
We refer the reader to the paper \cite{brian} for more information
on the current state of set-theoretic investigations of such phenomena.

The theory of tilings considers general normed spaces rather than $\R$, including
nonseparable spaces which are not even Polish. On the other hand
tilings considered in this theory are covers by closed sets, called \emph{tiles}, such that
the interiors of any two distinct tiles are nonempty and disjoint, but in general their boundaries may intersect. 
It is usual to consider tilings by bodies, that is closed sets which are closures of their interiors.
In particular, if the density of a normed space $X$ is less than continuum, there cannot be 
continuum many pairwise disjoint nonempty open subsets of $X$, so any  tiling  of such a space has cardinality less
than $2^\omega$. If the tiling is pairwise disjoint (called a disjoint tiling) then it yields many
disjoint covers by closed sets of any interval $[x, x']:=\{tx+(1-t)x': t\in [0,1]\}$ for $x, x'\in X$
 by intersecting the tiles with the interval. In this paper we focus on disjoint tilings. However, it 
should be mentioned that any normed space admits a general tiling by some convex bounded bodies
(\cite{zanco2}).

By a classical but surprising result of Klee~\cite{klee}, for every infinite cardinal number $\kappa$
satisfying $\kappa^\omega=\kappa$ there is a disjoint tiling of 
$$
\ell_1(\kappa):=\{x\in\R^\kappa:\|x\|<\infty\},\ \text{ with the norm }\ 
\|x\|=\sum_{\alpha\in \kappa}|x(\alpha)|,$$ 
by unit balls.

Of course, as already considered in \cite{klee} (p. 258), a natural question arises 
whether the hypothesis on $\kappa$ in Klee's result is necessary for the 
existence of tilings with at least similar geometric features.
First, let us note that the smallest cardinal satisfying the hypothesis is $2^\omega$,
no cardinal of countable cofinality satisfies it (see \cite{jech}) and the Singular Cardinal Hypothesis ({\sf SCH})
implies that all cardinals of uncountable cofinality above $2^\omega$ satisfy it (Theorem 5.22 of \cite{jech}).
So there may be three reasons why the hypothesis of Klee's result fails for
$\kappa$: the failure of {\sf SCH}, $\kappa<2^\omega$, or $\cf(\kappa)=\omega$.
We investigate the latter two.
The first seems to be completely left out in the literature, and we leave it as well.

For $\kappa<2^\omega$ already Klee uses Sierpi\'nski's result to conclude  in \cite{klee} that
no separable normed space admits a nontrivial\footnote{By a nontrivial tiling or cover
we will mean one containing at least two elements} disjoint tiling (Proposition 3.1 of \cite{klee}) or
that no normed space of density less than $2^\omega$ admits a nontrivial disjoint tiling
by convex sets  (Proposition 3.3 of \cite{klee}). It immediately follows that
when $[0,1]$ cannot be covered by more than one and less than $2^\omega$ pairwise
disjoint closed sets (which is consistent,
see e.g. \cite{brian}), then
no normed space of density less than $2^\omega$ admits a nontrivial disjoint tiling.
Very recently, De Bernardi, Russo, Sezgek and Somaglia proved in \cite{russo2} that for
$\kappa<2^\omega$ the space $\ell_1(\kappa)$ does not admit any tiling by balls (disjoint or not) .
Our contribution to the above line of research is on the other hand positive and
can be considered an uncountable-dimensional version of Stern's result from \cite{stern}.
A subset $K$ of a 
vector space is called \emph{star-shaped} if there is $x\in K$ such that for every $y\in K$, the
interval $[x, y]$ is included in $K$.

\begin{theorem}\label{thm:mainbelowc} For $\kappa=\omega_1$ or $\kappa=\omega_2$
and an infinite cardinal $\lambda>\kappa$ of uncountable cofinality
it is consistent that $2^\omega=\lambda$ and $\ell_1(\kappa)$ admits
a  normal\footnote{A tiling is normal if there are $0<r<R$ such that all tiles contain
a ball of radius $r$ and are contained in a ball of radius $R$.} 
disjoint tiling by star-shaped bodies of the form $X+B$ where $X$
is norm compact and $B$ is the unit ball of $\ell_1(\kappa)$.
\end{theorem}
\begin{proof} See section 3.5.
\end{proof}
We note that  Klee's tiling satisfies all the properties of our tiling and additionally
has sets $X$ equal to singletons. On the other hand
we prove a negative {\sf ZFC} result concerning cardinals of countable cofinality:

\begin{theorem}\label{maincofinal} Suppose that $\kappa$ is an uncountable cardinal
of countable cofinality. Then $\ell_1(\kappa)$ does not admit any normal disjoint  tiling  
by bodies of the form $X+B$, where $X$ is closed and norm-separable and $B$ is the unit ball
of $\ell_1(\kappa)$.
\end{theorem}
\begin{proof} Suppose that there is such a tiling.
Since all sets $X+B$ in the tiling are bodies, they are closed and hence
they are of the form $\mathrm{B}(X, 1)$ by Lemma~\ref{lem:closedtube}. 
By Lemma~\ref{lem:normalcard} the tiling has cardinality $\kappa$,
contradicting Proposition \ref{cfomega}.
\end{proof}
The authors were informed by T. Russo that he, C. A. De Bernardi and J. Somaglia obtained
independently an analogous negative result in the case of disjoint tiling by balls.

Now let us discuss connections of these results with approximation theory which was
the main motivation of \cite{klee}.
In approximation theory for a given metric space $M$ one considers 
\emph{proximinal sets} $K\subseteq M$, i.e., sets such that every $x\in M$
admits a point of $K$ which is closest in $K$ to $x$. The map $P_K: M\rightarrow \wp(K)$ sending $x$
to the set of all such closest points in $K$ is called the \emph{metric projection}.
A subset $C$ of a metric space $M$ is called a \emph{Chebyshev set} if for every $x\in M$ 
there is exactly one $y\in C$ which is the closest point of $C$ to $x$.
So a proximinal set $K$ is a Chebyshev set if and only if
the metric projection $P_K$ is a function from $M$ into $K$.

Note that the centers of balls of a disjoint tiling by unit balls like Klee's one  form a discrete
Chebyshev set. 
In general, if $K$ is Chebyshev, the fibers of the metric projection $P_K$ 
  called \emph{Voronoi cells}  of $K$ form a disjoint cover of $M$ by
closed  \emph{star-shaped} sets (cf. Section 2 of \cite{russo1}). 

Discrete nontrivial Chebyshev sets provide
an extreme example of unique best approximation which is not continuous (as the
ambient Banach space is connected). To the best of our knowledge, Klee's examples
are still the only known source of discrete Chebyshev sets in normed spaces.
It could be added that one of the well-known open problems in approximation theory
is whether an infinite-dimensional Hilbert space admits a nonconvex Chebyshev set
(\cite{moors}).

Arguments as above in the case of a discrete Chebyshev set imply that
given the tiling $\{X_i+B: i\in I\}$ of
Theorem \ref{thm:mainbelowc} the set $K=\bigcup\{X_i: i\in I\}$ is proximinal,
but best approximations cannot be chosen continuously, i.e., $P_K$ admits no continuous selection.
Also our tiles of Theorem \ref{thm:mainbelowc} share with the Voronoi cells
of discrete Chebyshev sets such properties as being star-shaped, having
nonempty interiors and being closed.
However, despite these similarities, Theorem \ref{thm:mainchebyshev} 
(obtained by a digital assistant and included with the proof in the Appendix for the convenience of the reader)
shows that there are no discrete Chebyshev sets in
$\ell_1(\kappa)$ for $\kappa<2^\omega$ and so our sets $X$ from Theorem \ref{thm:mainbelowc}
cannot be replaced by points even consistently.

The structure of the paper is the following. 
Section 2 contains preliminaries.
In Section~\ref{sec:sacks} we prove Theorem \ref{thm:mainbelowc}, using
side-by-side Sacks forcing and adapting Klee's construction for points to   forcing names
for points of $\ell_1(\kappa)$ obtaining the sets $X$ from Theorem \ref{thm:mainbelowc}.
In Section~\ref{sec:cofw} we prove Proposition \ref{cfomega}, and hence
Theorem \ref{maincofinal}.
The Appendix contains the machine-discovered above-mentioned result  on Chebyshev sets with the proof.
 The results of this paper generate
the following natural questions.

\begin{question}$ $
\begin{enumerate}
\item Is it consistent that there is a disjoint normal tiling of $\ell_1(\omega_n)$ 
by  bodies  of the form $X+B$ where $X$ is separable and $B$ the unit ball
while $\omega_n<2^\omega$ for $n\in\N\setminus\{0, 1, 2\}$?
\item Is it consistent that there is a normal disjoint tiling of $\ell_1(\omega_\omega)$
by  star-shaped bodies?
\end{enumerate}
\end{question}
Concerning the first question one could conjecture that the methods like those
in \cite{morasses} could provide the positive answer in the case of $n=3$.
On the other hand Theorem \ref{maincofinal} provides the limitation on
results of this type.

\section{Preliminaries}\label{sec:prelim}

By $\N$ denote the set of nonnegative integers and
$\N_+=\N\setminus\{0\}$.
Let $\kappa$ be a cardinal. 
The restriction of $x\in\ell_1(\kappa)$ to a set $A\subseteq\kappa$ is denoted by $x|A$.
For $x\in\ell_1(\kappa)$, $\supp(x)$ denotes its support, and for
$X\subseteq\ell_1(\kappa)$ we put
$\supp(X)=\bigcup_{x\in X}\supp(x)$. For
$S\subseteq\kappa$ we identify $\ell_1(S)$ with the subspace of
$\ell_1(\kappa)$ consisting of vectors supported by $S$.
By $1_\alpha$ we denote the unit vector 
whose support is $\{\alpha\}$ for an ordinal $\alpha$.
By $\omega_1, \omega_2, 2^\omega$ we will denote the first
two uncountable cardinals and the cardinality of the continuum, respectively.

The distance from a point $x$ to a set $X$ and the distance between sets
$X$ and $Y$ are denoted by $\dist(x,X)$ and $\dist(X,Y)$, respectively,
and $\diam(X)$ denotes the diameter of $X$.
For $r>0$, $\mathrm{B}(X,r)=\{x\in \ell_1(\kappa): \dist(x, X)\leq r\}$.
Also $B(\ell_1(\kappa))=\mathrm{B}(\{0\},1)$.
The notions of tilings, normal tiling, 
body, Chebyshev set, proximinal set were defined in the Introduction.
Our terminology and notation concerning Banach spaces, set theory and general
topology are standard and follow, respectively, \cite{fabian}, \cite{jech},
and~\cite{engelking}.

\begin{lemma}\label{support}
A subset $X$ of $\ell_1(\kappa)$ is separable if and only if $\supp(X)$ is countable. 
\end{lemma}

\begin{lemma}\label{lem:normalcard}
Suppose that $X$ is a Banach space of infinite density $\kappa$.
Every normal tiling of $X$ has cardinality $\kappa$.
\end{lemma}

\begin{proof}
Since the interiors of distinct elements of a tiling
are disjoint and nonempty the cardinality cannot be bigger than $\kappa$.
The normality yields a cover of $X$ by $\{\mathrm{B}(\{x_i\}, R): i\in I\}$, where
the cardinality of $I$ and the cardinality of the tiling are the same and infinite and $R>0$.
For every $y\in X$ and $n\in \N$ there is $i\in I$ such that $\|nRy-x_i\|\leq R$, so
$\{x_i/nR: i\in I, n\in \N\setminus\{0\}\}$ forms a dense subset of $X$ of cardinality at most 
 the cardinality of $I$.
\end{proof}

\begin{lemma}\label{lem:closedtube}
Suppose that $X\subseteq\ell_1(\kappa)$ is nonempty and closed and $B$ is
the unit ball of $\ell_1(\kappa)$. Then
$
\mathrm{B}(X,1)=X+B
$
if and only if $X+B$ is closed.
\end{lemma}

\begin{proof}
It is enough to  note that
$
\overline{X+B}=\mathrm{B}(X,1).
$
Indeed, $X+B\subseteq\mathrm{B}(X,1)$, and $\mathrm{B}(X,1)$ is closed.
For the other inclusion note that for any  $y\in\mathrm{B}(X,1)$ there is a sequence
of points in $X+B$  which converges to $y$.
\end{proof}

\begin{lemma}\label{B} Suppose that $X\subseteq\ell_1(\kappa)$ is compact
and $\diam(X)\leq 1$.
Then $\mathrm{B}(X,1)=X+B(\ell_1(\kappa))$ is a bounded star-shaped body.
\end{lemma}

\begin{proof}
Since $X$ is compact, $X+B(\ell_1(\kappa))$ is closed, so
$\mathrm{B}(X,1)=X+B(\ell_1(\kappa))$ by Lemma~\ref{lem:closedtube}.
It is bounded since $X$ is bounded, and it is a body since every
$y=x+z$ with $x\in X$ and $\|z\|\leq1$ is the limit of the interior
points $x+tz$, where $0\leq t<1$.

Finally, if $y\in X+B(\ell_1(\kappa))$, choose $x_1\in X$ such that
$\|y-x_1\|\leq1$. For every $x_2\in X$ and $0\leq t\leq1$ we have
$\|tx_2+(1-t)y-x_1\|\leq t\|x_2-x_1\|+(1-t)\|y-x_1\|\leq1$,
so $[x_2,y]\subseteq X+B(\ell_1(\kappa))$.
\end{proof}

\section{Positive consistency result below the continuum}\label{sec:sacks}

Throughout this section $\kappa=\omega_1$ or $\kappa=\omega_2$.

The script font used below marks sets determined by fixed ground-model
codes.
Thus, $\cX(S,m,A,\I)$ is interpreted in the model under consideration according to
Definition~\ref{df:basic}, an $S$-center $\cY$ is interpreted from the ground-model data
occurring in its representation in Definition~\ref{dfn:Scenter}, and a translated center
$\cX$ is interpreted from the code of $\cY_{\cX}$ together with $\alpha_{\cX}$.
We regard $\cX_0=\{1_0\}$ as a fixed trivial code. The same symbol may therefore denote
different sets in different forcing extensions. We suppress check accents on ground-model
parameters in forcing statements and use dots only for genuine forcing names, such as
$\dot{x}$; in particular, no dot is attached to $\cX$ or $\cY$.

Each coded set used here has countable support and can be viewed as a coded closed subset
of a Polish space $\ell_1(T)$ for an appropriate countable $T\subseteq\kappa$.
 By Mostowski absoluteness
for $\Sigma^1_1$ statements, see \cite[Theorem~25.4]{jech}, the existence statements about
points in these sets used below are absolute. In particular, nonemptiness and inequalities
involving $\dist$ or $\diam$ and rational bounds are preserved between the ground model and
forcing extensions.

\subsection{Basic closed approximations}

\begin{definition}\label{df:basic}
Let $m\in\N_+$, $S\subseteq\kappa$ be a countable set, $A\subseteq S$ be a finite set,
and
$\I=(I_\alpha)_{\alpha\in A}$ be an indexed family of closed nonempty intervals in
$\R$ with rational endpoints such that $\sum_{\alpha\in A}\diam(I_\alpha)\leq
\frac{1}{m}$ and $0\notin\bigcup\I$.
Let $\cX(S, m, A, \I)$ be the set of all $x\in \ell_1(\kappa)$ satisfying:
\begin{enumerate}
\item $\supp(x)\subseteq S$;
\item $x|A\in \prod_{\alpha\in A}I_\alpha$;  
\item $\|x|(S\setminus A)\|\leq \frac{1}{m}$.
\end{enumerate}
\end{definition}

\begin{lemma}\label{ctbl-approx}
For a fixed countable set $S\subseteq\kappa$, there are countably many $\cX(S,m,A, \I)$
sets in $\ell_1(\kappa)$.
\end{lemma}

\begin{lemma}\label{fct:IBm}
For each set $\cX(S,m, A,\I)$ in $\ell_1(\kappa)$, we have $\diam(\cX(S,m,
A,\I))\leq \frac{3}{m}.$
\end{lemma}

\begin{lemma}\label{intpro-closed} Each set $\cX(S,m, A,\I)$ in $\ell_1(\kappa)$ is
norm-closed.
\end{lemma}
\begin{proof}
The set of all $S$-supported vectors is closed in $\ell_1(\kappa)$.
For every $\alpha\in A$, the coordinate functional $x(\alpha)$ is continuous, and so is
$\|x|(S\setminus A)\|$.
Since the intervals $I_\alpha$ are closed, it follows directly from the definition
that $\cX(S,m,A,\I)$ is closed.
\end{proof}

\subsection{Sacks approximations}
For an infinite cardinal $\lambda$ let
$\SSS^{\lambda}$ be the countable-support  side-by-side product of $\lambda$ many copies of the Sacks
forcing. We will also consider $\SSS^C$ for a countable set of ordinals $C$, 
defined analogously to $\SSS^\lambda$.
Our terminology concerning $\SSS^\lambda$ follows Baumgartner~\cite{failure}, with one
exception: for a finite set $F\subseteq\omega$, instead of writing $p\leq q\pmod{F,n}$ we
write $p\leq_{F,n}q$.
In our applications we use only the relation $p\leq_{n,n}q$, where the first coordinate
denotes the set $\{0,\ldots,n-1\}$.
We will also use without further mention the standard continuous reading of names for 
side-by-side products of Sacks forcing. 
In the terminology of Baumgartner, the corresponding determination of names is 
given by Theorem 1.11 of \cite{failure};
 see also Section II of \cite{life}
 and Lemma 4 of \cite{vera} for formulations in terms
 of continuous conditions and continuous reading of names. In particular, after strengthening 
a condition, a name for a countable sequence of ground-model objects is determined by
 a countable ground-model code and depends on countably many Sacks coordinates.
\begin{lemma}\label{lem:IBm}
Let $p\in\SSS^{\omega}$, $\dot{x}$ be an $\SSS^{\omega}$-name for an element of
$\ell_1(\kappa)$ such that $\SSS^{\omega}\forces\dot{x}\neq 0$, and let
$S\subseteq\kappa$ be a countable set such that
$$
p\forces\ \supp(\dot{x})\subseteq S\ \wedge\ \sup\supp(\dot{x})=\sup S.
$$
For every $m\in\N_+$ and $\tilde{\alpha}\leq\sup S$ such that
$p\forces\supp(\dot{x})\cap[\tilde{\alpha},\kappa)\neq\emptyset$, there is a condition $q\leq p$ and a
set $\cX(S,m, A, \I)$ in $\ell_1(\kappa)$ such that
$$
\tilde{\alpha}\leq\max A\quad\text{and}\quad
q\forces \dot{x} \in \cX(S, m, A, \I).
$$
\end{lemma}
\begin{proof}
First, by strengthening $p$ if needed and keeping the same notation, choose
$\beta\in S\cap[\tilde{\alpha},\kappa)$ such that
$p\forces\beta\in\supp(\dot{x})$.
Since $p\forces\dot{x}\in\ell_1(\kappa)$ and $\supp(\dot{x})\subseteq S$, by
strengthening $p$ again we may choose a finite set $A\subseteq S$ containing
$\beta$ such that
$$
p\forces\ A\subseteq\supp(\dot{x})\ \wedge\ \sum_{\alpha\in S\setminus
A}|\dot{x}(\alpha)|\leq \frac{1}{m}.
$$
In particular, $\tilde{\alpha}\leq\beta\leq\max A$.
Since $A$ is finite and $p\forces\dot{x}(\alpha)\neq0$ for every $\alpha\in A$,
by strengthening $p$ if needed choose a condition $q\leq p$ and an indexed
family $\I=(I_\alpha)_{\alpha\in A}$ of closed intervals in $\R$ with rational
endpoints such that
$$
0\notin\bigcup_{\alpha\in A}I_\alpha,\quad
\sum_{\alpha\in A}\diam(I_\alpha)\leq\frac{1}{m},\quad\text{and}\quad
q\forces \dot{x}|A\in\prod_{\alpha\in A}I_\alpha.
$$
We conclude that $q\forces\dot{x}\in \cX(S,m,A,\I)$.
\end{proof}

\begin{lemma}[$(F,n)$-decision, cf.~Lemmas~1.5--1.8 of \cite{failure}]\label{lem:Fndec}
Let $p\in\SSS^{\omega}$, $n\in\N$, and $D\subseteq \SSS^{\omega}$ be an
open dense set.
Then there is a condition $q\leq_{n,n} p$ such that $q_\sigma\in D$ for all
$\sigma\in\lev(n,n,q)$.
\end{lemma}

\subsection{Centers}
Let $(E_n)_{n\in\N}$ be a sequence of finite nonempty sets of ordinal numbers.
We write $\lim_{n\to\infty}E_n=\beta$, for some ordinal number $\beta$, if $\max E_n\leq
\beta$ for all $n\in\N$ and $\lim_{n\to\infty}\min E_n=\beta$. So in
particular $\lim_{n\to\infty}E_n=\beta$ if $E_n=\{\beta\}$ for each $n\in\N$.

\begin{definition}\label{dfn:Scenter}
Let $p\in\SSS^{\omega}$, $\dot{x}$ be a name for an element of $\ell_1(\kappa)$, and
$S\subseteq\kappa$ be a countable set.
A set $\cY\subseteq \ell_1(\kappa)$ is an \emph{$S$-center for $(p,\dot{x})$} if there are 
\begin{enumerate}
\item an increasing sequence $(m_n)_{n\in\N}$ of positive natural numbers,
\item sets $\cX(S, m_n, A_\sigma, \I_\sigma)$ 
for all $\sigma\in \lev(n, n, p)$ and $n\in \N$,
\item a sequence $(E_n)_{n\in\N}$ of finite nonempty subsets of $S$ with
$\lim_{n\to\infty}E_n=\sup S$,
\item a sequence $(\epsilon_n)_{n\in\N}$ of positive rational numbers,
\end{enumerate}

such that

\begin{equation}
\cY=\bigcap_{n\in \N}\bigcup_{\sigma\in \lev(n, n, p)}\cX(S, m_n, A_\sigma,
\I_\sigma),
\end{equation}
\begin{equation}
p\forces\dot{x}\in \cY
\end{equation}
\begin{equation}
\label{eq:Eeps}
\forall y\in \cY\ \forall n\in\N\ \exists \alpha\in E_n\ |y(\alpha)|>\epsilon_n.
\end{equation}

A set is an \emph{$S$-center} if it is an $S$-center for some $(p,\dot{x})$ with
$p\in\SSS^{\omega}$ and $\SSS^{\omega}$-name $\dot{x}$.
We say that an $S$-center is \emph{witnessed} by the objects listed in items~(1)--(4).

If $\cY$ is an $S$-center with $\diam(\cY)<1$ and $\alpha\in\kappa\setminus S$, then
$\cX:=\cY+1_\alpha$ is called a \emph{translated center}. For such a representation we write
$S_{\cX}:=S$, $\cY_{\cX}:=\cY$, and $\alpha_{\cX}:=\alpha$.
\end{definition}

In applications, we may specify only the relevant witnesses.

\begin{lemma}\label{lem:Scenter}
Using the notation from Definition~\ref{dfn:Scenter}, the following assertions hold.
\begin{enumerate}
\item $\supp(\cY)\subseteq S$.
\item An $S$-center is a nonempty norm-compact set in $\ell_1(\kappa)$.
\item For a fixed countable set $S\subseteq\kappa$ there are at most $2^\omega$ many $S$-centers.
\end{enumerate}
\end{lemma}

\begin{proof}
Items~(1) and~(3) follow directly from the definition and Lemma~\ref{ctbl-approx}.
We prove~(2).
By~(1) and Lemma~\ref{intpro-closed}, the set $\cY$ is norm-closed
in $\ell_1(S)$, and thus in $\ell_1(\kappa)$.
Since $p\forces \dot{x}\in \cY$, nonemptiness follows from  the absoluteness discussed
at the beginning of this section.
The set $\cY$ is compact since it is closed in a complete metric space and totally bounded by 
Definition \ref{dfn:Scenter} and Lemma \ref{fct:IBm} (Theorem 4.3.29 of \cite{engelking}).
\end{proof}

\begin{lemma}\label{lem:IBmX}
Assume that $p\in \SSS^{\omega}$, $\dot{x}$ is an $\SSS^{\omega}$-name for an
element of $\ell_1(\kappa)$, $S\subseteq \kappa$ is a countable set, and $\cX$ is either $\cX_0=\{1_0\}$ or a
translated center, such that 
$$
p\forces \dot{x}\neq0\ \wedge\ \supp(\dot{x})\subseteq S\ \wedge\ \sup\supp(\dot{x})=\sup S\ \wedge\ 
\dist(\dot{x}, \cX)>1.
$$
Let $n\in\N$, $k\in\N_+$, and let $\tilde{\alpha}\leq\sup S$ be such that
$p\forces\supp(\dot{x})\cap[\tilde{\alpha},\kappa)\neq\emptyset$.
Then there is a condition $q\leq_{n,n} p$ and sets $\cX(S, m, A_\sigma, \I_\sigma)$
for all $\sigma\in \lev(n, n, q)$ such that
\begin{equation}
m>k,
\end{equation}
\begin{equation}\label{eq:maxA}
\max A_\sigma\geq\tilde{\alpha}\text{ for all }\sigma\in\lev(n,n,q),
\end{equation}
\begin{equation}
q\forces \dot{x} \in \bigcup_{\sigma\in \lev(n,n,q)}\cX(S,m,A_\sigma,\I_\sigma),
\end{equation}
\begin{equation}\label{eq:dist}
\dist\Big(\bigcup_{\sigma\in \lev(n,n,q)}\cX(S,m,A_\sigma,\I_\sigma),\cX\Big)>1.
\end{equation}

In particular, there are a finite nonempty set $E\subseteq S\cap[\tilde{\alpha},\kappa)$ and
$\epsilon\in\Q_+$ such that
\begin{equation}
\forall y\in \bigcup_{\sigma\in\lev(n,n,q)}\cX(S,m,A_\sigma,\I_\sigma)\ 
\exists \alpha\in E\ |y(\alpha)|>\epsilon.
\end{equation}
\end{lemma}

\begin{proof}
Since $p\forces\dot{x}\neq0$, by changing the value of the name $\dot{x}$ outside $p$,
we may assume that $\SSS^\omega\forces\dot{x}\neq0$.
By Lemma~\ref{lem:Fndec} and the finiteness of $\lev(n,n,p)$, there is
$m>k$ and a condition $p'\leq_{n,n} p$ such that
\begin{equation}\label{eq:p'}
p'\forces \dist(\dot{x}, \cX)\geq 1+\frac{4}{m}.
\end{equation}
By Lemmata~\ref{lem:IBm} and~\ref{lem:Fndec}, there is a condition $q\leq_{n,n} p'$ and sets
$\cX(S, m, A_\sigma, \I_\sigma)$, for $\sigma\in\lev(n,n,q)$, such that $\max
A_\sigma\geq\tilde{\alpha}$ and
\begin{equation}\label{eq:qsigma}
q_\sigma\forces \dot{x}\in \cX(S, m, A_\sigma, \I_\sigma)
\end{equation}
for all $\sigma\in \lev(n,n,q)$.
Since $\lev(n,n,q)=\lev(n,n,p)$ and $\{q_\sigma:\sigma \in\lev(n,n,q)\}$ is a maximal
antichain below $q$, we have 
$$
q\forces \dot{x}\in \bigcup_{\sigma\in\lev(n,n,q)}\cX(S, m, A_\sigma, \I_\sigma).
$$

Fix $y\in \cX$, $\sigma\in\lev(n,n,q)$, and 
$z\in \cX(S,m,A_\sigma,\I_\sigma)$.
Applying~\eqref{eq:p'} with $q_\sigma\leq p'$ and Lemma~\ref{fct:IBm} with~\eqref{eq:qsigma} we
have
$$
q_\sigma\forces  \|\dot{x}-y\|\ge 1+\frac{4}{m}\quad\text{and}\quad q_\sigma\forces
\|\dot{x}-z\|\le \frac{3}{m}.
$$
It follows that 
$$
q_\sigma\forces
\|y-z\|
\ge \|\dot{x}-y\|-\|\dot{x}-z\|
\ge 1+\frac{4}{m}-\frac{3}{m}
=1+\frac{1}{m}.
$$
Thus, \eqref{eq:dist} holds.
For the last assertion
put $E:=\{\max A_\sigma:\sigma\in\lev(n,n,q)\}$
and choose 
$\epsilon\in\Q_+$ such that
$$
(-\epsilon,\epsilon)\cap \bigcup_{\sigma\in\lev(n,n,q)}\bigcup \I_\sigma=\emptyset.\qedhere
$$
\end{proof}

\begin{lemma}\label{lem:IBmXn}
Assume that $p\in \SSS^\omega$, $\dot{x}$ is an $\SSS^\omega$-name for an element 
of $\ell_1(\kappa)$, $S\subseteq \kappa$ is a countable set, and $\cX_n$ is either $\cX_0=\{1_0\}$ or a
translated center for every $n\in\N$, such that
$$
p\forces \dot{x}\neq0\ \wedge\ \supp(\dot{x})\subseteq S\
 \wedge\ \sup\supp(\dot{x})=\sup S\ \wedge\
\forall n\in\N\ \dist(\dot{x},\cX_n)>1.
$$
Then there are a condition $q\leq p$ and an $S$-center $\cY$ for $(q,\dot{x})$ such that
$\dist(\cY,\cX_n)>1$ for all $n\in\N$ and $\diam(\cY)<1$.
\end{lemma}

\begin{proof}

Fix a sequence $(\tilde{\alpha}_n)_{n\in\N}$ of ordinals such that
$\tilde{\alpha}_n\leq\sup S$ and
$p\forces\supp(\dot{x})\cap[\tilde{\alpha}_n,\kappa)\neq\emptyset$ for every $n\in\N$, and,
for every sequence $(E_n)_{n\in\N}$ of finite nonempty sets with $E_n\subseteq
S\cap[\tilde{\alpha}_n,\kappa)$, we have $\lim_{n\to\infty}E_n=\sup S$.
Such a sequence exists: if $S\setminus\{\sup S\}$ is nonempty and cofinal in $\sup S$, take
$(\tilde{\alpha}_n)_{n\in\N}$ cofinal in $\sup S$ with $\tilde{\alpha}_n<\sup S$; otherwise take
the constant sequence $\tilde{\alpha}_n=\sup S$, in which case the assumptions imply
$p\forces\sup S\in\supp(\dot{x})$.
Applying Lemma~\ref{lem:IBmX}, define a fusion sequence 
$((q_n,n,m_n))_{n\in\N}$, where $(m_n)_{n\in\N}$ is an increasing
sequence of positive natural numbers with $m_0>3$, sets $\cX(S,m_n,A_\sigma,\I_\sigma)$
for all $\sigma\in\lev(n,n,q_n)$, finite nonempty sets $E_n\subseteq 
S\cap[\tilde{\alpha}_n,\kappa)$,
and positive rational numbers $\epsilon_n$ such that $q_0=p$, $q_{n+1}\leq_{n,n}q_n$,
$$
q_{n+1}\forces \dot{x}\in
\bigcup_{\sigma\in\lev(n,n,q_n)}\cX(S,m_n,A_\sigma,\I_\sigma),
$$
$$
\dist\Big( \bigcup_{\sigma\in \lev(n, n, q_n)}\cX(S, m_n, A_\sigma, \I_\sigma)\,,\,
\cX_n\Big)> 1,
$$
and
$$
\forall y\in \bigcup_{\sigma\in\lev(n,n,q_n)}\cX(S,m_n,A_\sigma,\I_\sigma)\ 
\exists \alpha\in E_n\ |y(\alpha)|>\epsilon_n
$$
for all $n\in\N$.
Put $q:=\bigcap_{n\in\N_+}q_n$.
Since $q\leq_{n,n}q_n$ for every $n\in\N$, we have $\lev(n,n,q)=\lev(n,n,q_n)$.
Let
$$
\cY:=\bigcap_{n\in \N}\bigcup_{\sigma\in \lev(n, n, q)}\cX(S, m_n, A_\sigma,
\I_\sigma).
$$
Then $q\leq p$ and $q\forces \dot{x}\in \cY$.
Moreover, by the choice of $(\tilde{\alpha}_n)_{n\in\N}$, we have
$\lim_{n\to\infty}E_n=\sup S$.
Thus $\cY$ is an $S$-center for $(q,\dot{x})$.
Finally, for every $n\in\N$,
$$
\cY\subseteq \bigcup_{\sigma\in \lev(n,n,q)}\cX(S,m_n,A_\sigma,\I_\sigma)\subseteq
\bigcup_{\sigma\in \lev(n,n,q_n)}\cX(S,m_n,A_\sigma,\I_\sigma),
$$
and hence $\dist(\cY,\cX_n)>1$.

Since $\lev(0,0,q)=\{\emptyset\}$, we have $\cY\subseteq
\cX(S,m_0,A_\emptyset,\I_\emptyset)$.
Since $m_0>3$, Lemma~\ref{fct:IBm} yields $\diam(\cY)\leq\frac{3}{m_0}<1$.
\end{proof}

\subsection{The auxiliary forcing $\mathbb{X}$}

In this subsection we work with  $\kappa=\omega_2$.

\begin{definition}
Let $\mathbb{X}$ be the set consisting of all nonempty countable families $\X$ such that
\begin{enumerate}
\item $\cX_0=\{1_0\}\in\X$, and each $\cX\in\X\setminus\{\cX_0\}$ is a translated center,
\item for each $\cX\in\X\setminus\{\cX_0\}$, we have $\sup S_{\cX}<\alpha_{\cX}<\sup S_{\cX}+\omega_1$,
\item $\alpha_{\cX}\neq\alpha_{\cZ}$ for distinct $\cX,\cZ\in \X\setminus\{\cX_0\}$,
\item $\dist(\cX,\cZ)>2$ for distinct $\cX,\cZ\in \X$.
\end{enumerate}
For $\X,\X'\in\mathbb{X}$ we write $\X\leq \X'$ if
$\X\supseteq \X'$.
\end{definition}

\begin{lemma}\label{lem:sigma}
The forcing $\mathbb{X}$ is $\sigma$-closed.
\end{lemma}

\begin{lemma}[{\sf CH}]\label{lem:kapcc}
The forcing $\mathbb{X}$ is $\omega_2$-cc.
\end{lemma}

\begin{proof}
Let $\{\X_\beta:\beta<\omega_2\}$ be a family of $\mathbb{X}$-conditions.
Define
$$
\tilde S_\beta:=\bigcup_{\cX\in\X_\beta\setminus\{\cX_0\}}S_{\cX}\cup
\{\alpha_{\cX}:\cX\in\X_\beta\setminus\{\cX_0\}\}
$$
for all $\beta<\omega_2$.
By $\mathsf{CH}$  and by the corresponding $\Delta$-system Lemma for  countable sets, 
we may assume that the family
$\{\tilde S_\beta:\beta<\omega_2\}$ has a root $\Delta\subseteq\delta$, for some
$\delta<\omega_2$, and that $\tilde S_\beta\cap\delta=\Delta$ for all
$\beta<\omega_2$.
For each $\beta<\omega_2$, put
$\mathcal{R}_\beta:=\{\cX_0\}\cup\{\cX\in\X_\beta\setminus\{\cX_0\}:S_{\cX}\subseteq\delta\}$.
By Lemma~\ref{lem:Scenter}(3),  by item~(2) in the definition of $\mathbb{X}$  and by
$\mathsf{CH}$, there are at most $\omega_1$ translated centers $\cX$ such that
$S_{\cX}\subseteq\delta$, and, together with the fixed set $\cX_0$, there are at most
$\omega_1$ possibilities for the countable family $\mathcal{R}_\beta$.
Passing to a subfamily of size $\omega_2$, we may assume that
$\mathcal{R}_\beta=\mathcal{R}$ for all $\beta<\omega_2$.

If $\X_\beta=\mathcal{R}$ for some $\beta<\omega_2$, then $\X_\beta$ is compatible
with every condition in the family.
Thus, we may assume that $\X_\beta\setminus\mathcal{R}\neq\emptyset$ for all
$\beta<\omega_2$.
Put $\gamma:=\sup\tilde S_0$.
Since $|[\delta,\gamma+1)|\leq\omega_1$ and the sets
$\tilde S_\beta\setminus\Delta$ are pairwise disjoint, there is $\beta>0$ such that
$$
\tilde S_\beta\cap[\delta,\gamma+1)=\emptyset.
$$
We claim that $\X_0\cup\X_\beta\in\mathbb{X}$.
It remains to consider $\cX\in\X_\beta\setminus\mathcal{R}$ and
$\cZ\in\X_0\setminus\mathcal{R}$.
Fix such $\cX$ and $\cZ$, and write $\cX=\cY_{\cX}+1_{\alpha_{\cX}}$ and
$\cZ=\cY_{\cZ}+1_{\alpha_{\cZ}}$.

Since $\cX\notin\mathcal{R}$, we have $S_{\cX}\not\subseteq\delta$.
By the choice of $\beta$, it follows that $\sup S_{\cX}>\gamma$.
Let $(E_l^{\cX})_{l\in\N}$ and $(\epsilon_l^{\cX})_{l\in\N}$ witness the
$S_{\cX}$-center $\cY_{\cX}$, and thus satisfy~\eqref{eq:Eeps}.
Choose $l\in\N$ such that $\min E_l^{\cX}>\gamma$.

Fix $x\in \cX$ and $z\in \cZ$, and let $y_x\in \cY_{\cX}$ and $y_z\in \cY_{\cZ}$ satisfy
$x=y_x+1_{\alpha_{\cX}}$ and $z=y_z+1_{\alpha_{\cZ}}$.
Since $\cZ\notin\mathcal{R}$, we have $\sup S_{\cZ}\geq\delta$.
Moreover, $\alpha_{\cZ}\in\tilde S_0$, and thus
$\delta<\alpha_{\cZ}\leq\gamma$.
It follows from the choice of $\beta$ that $\alpha_{\cZ}\notin S_{\cX}$.
Also, $\alpha_{\cX}>\sup S_{\cX}>\gamma\geq\alpha_{\cZ}$, and hence
$x(\alpha_{\cZ})=0$, while $z(\alpha_{\cZ})=1$.

By~\eqref{eq:Eeps}, there is $\alpha\in E_l^{\cX}$ such that
$|y_x(\alpha)|>\epsilon_l^{\cX}$.
Since $\alpha>\gamma$, we have $z(\alpha)=0$.
Also, $\alpha_{\cX}>\gamma$, so $z(\alpha_{\cX})=0$, while
$x(\alpha_{\cX})=1$.
Therefore
$$
\|x-z\|\geq
|x(\alpha_{\cZ})-z(\alpha_{\cZ})|+|x(\alpha)-z(\alpha)|+
|x(\alpha_{\cX})-z(\alpha_{\cX})|>1+\epsilon_l^{\cX}+1.
$$
It follows that $\dist(\cX,\cZ)>2$.
Thus, $\X_0\cup\X_\beta\in\mathbb{X}$.
\end{proof}

\begin{lemma}\label{lem:dense}
Let $p\in\SSS^{\omega}$ and $\dot{x}$ be an $\SSS^{\omega}$-name for an element of
$\ell_1(\omega_2)$.
Then the set 
$$
\mathcal{D}_{p,\dot{x}}:=\{\X\in \mathbb{X}:\quad
\exists \cX\in\X\ \exists q\leq p\ q\forces \dist(\dot{x}, \cX)\leq 1\}
$$
is dense in $\mathbb{X}$.
\end{lemma}

\begin{proof}
Fix $\X\in\mathbb{X}\setminus \mathcal{D}_{p,\dot{x}}$ and write
$\X=\{\cX_n:n\in \N\}$.
By the definition of $\mathcal{D}_{p,\dot{x}}$,
$p\forces \dist(\dot{x}, \cX_n)>1$ for all $n\in\N$. 
Since $\cX_0\in\X$ and $\dist(0,\cX_0)=1$, it follows that $p\forces\dot{x}\neq0$.
Strengthening $p$, if needed, we may assume that there is a countable set $S\subseteq\omega_2$
such that 
$$
p\forces\supp(\dot{x})\subseteq S\ \wedge\ \sup\supp(\dot{x})=\sup S.
$$
Applying Lemma~\ref{lem:IBmXn} to the condition $p$ and to the sets $S$ and $\cX_n$, we get a
condition $q\leq p$ and an $S$-center $\cY$ for $(q,\dot{x})$ such that
$\dist(\cY,\cX_n)>1$ for all $n\in\N$.
Pick $\alpha\in(\sup S,\sup S+\omega_1)$ outside
$\{0\}\cup\bigcup_{\cZ\in\X\setminus\{\cX_0\}}(S_{\cZ}\cup\{\alpha_{\cZ}\})$.
Such an $\alpha$ exists since this set is countable. Put $\cX:=\cY+1_\alpha$.
It follows that $q\forces \dot{x}\in \cY$, and thus $q\forces
\dist(\dot{x},\cX)\leq 1$.
By Lemma~\ref{lem:Scenter}(1) and the choice of $\alpha$, all elements of $\cY$ and
$\cZ$ vanish at $\alpha$ for every $\cZ\in\X$. Hence
$\dist(\cX,\cZ)=\dist(\cY,\cZ)+1>2$ for all $\cZ\in\X$.
We conclude that $\X\cup\{\cX\}\in\mathcal{D}_{p,\dot{x}}$.
\end{proof}

\subsection{{Proof of Theorem~\ref{thm:mainbelowc}}}

\begin{proposition}\label{prop:tilingfamily}
Let $\kappa=\omega_1$ or $\kappa=\omega_2$, and let $\lambda>\kappa$ be a cardinal number of
uncountable cofinality in a model $V$ satisfying $\mathsf{GCH}$.
There is a model $W$ containing $V$ and satisfying $\mathsf{CH}$ with
$(\SSS^\lambda)^V=(\SSS^\lambda)^W$ such that in $W$ we have
$\lambda^\omega=\lambda>\kappa$ and there is a family $\tilde{\X}$ of subsets of
$\ell_1(\kappa)$ with the following properties:
\begin{enumerate}
\item $\cX_0\in\tilde{\X}$, and each $\cX\in\tilde{\X}\setminus\{\cX_0\}$ is a translated center,
\item $\dist(\cX,\cX')>2$ for distinct $\cX,\cX'\in\tilde{\X}$,
\item for every $p\in\SSS^{\omega}$ and every $\SSS^{\omega}$-name $\dot{x}$ for an
element of $\ell_1(\kappa)$, there are $q\leq p$ and $\cX\in\tilde{\X}$ such that
$q\forces \dist(\dot{x},\cX)\leq 1$.
\end{enumerate}
\end{proposition}

\begin{proof}
\textbf{Case 1:} Assume that $\kappa=\omega_1$.
In this case, the auxiliary forcing $\mathbb{X}$ is not needed; instead, we construct the
required family by recursion of length $\omega_1$.
Put $W:=V$.
By continuous reading of names and $\mathsf{CH}$, there 
is a family $\mathcal{N}$ of $\SSS^{\omega}$-names for elements of 
$\ell_1(\omega_1)$ such that $|\mathcal{N}|=2^\omega=\omega_1$ and for every
$\SSS^{\omega}$-name $\dot{x}$ for an element of $\ell_1(\omega_1)$ there is
$\dot{y}\in\mathcal{N}$ such that $
\SSS^{\omega}\forces\dot{x}=\dot{y}$.
Enumerate $\SSS^{\omega}\times\mathcal{N}$ as $\{(p_\alpha,
\dot{x}_\alpha):0<\alpha<\omega_1\}$.
Recursively on $0<\alpha<\omega_1$ define sets $\cX_\alpha$ in $\ell_1(\omega_1)$
and conditions $q_\alpha\leq p_\alpha$, where each $\cX_\alpha$ is either $\cX_0$ or a translated center, such that
\begin{enumerate}
\item $\dist(\cX_\alpha, \cX_\beta)>2$ for all $\beta<\alpha$ with $\cX_\beta\neq \cX_\alpha$,
\item $\diam(\cX_\alpha)<1$,
\item there is $\beta\leq \alpha$ such that $q_\alpha\forces \dist(\dot{x}_\alpha,
\cX_\beta)\leq 1$.
\end{enumerate}
Fix an ordinal number $0<\alpha<\omega_1$ and assume that sets $\cX_\beta$ in
$\ell_1(\omega_1)$, each equal to $\cX_0$ or a translated center, have been defined for all
$\beta<\alpha$, and conditions $q_\beta\leq p_\beta$ have been defined for all
$0<\beta<\alpha$.
If there are $\beta<\alpha$ and $r\leq p_\alpha$ such that $r\forces \dist(\dot{x}_\alpha,
\cX_\beta)\leq 1$, then put $q_\alpha:=r$ and $\cX_\alpha:=\cX_\beta$.
If not, then
$p_\alpha\forces \dist(\dot{x}_\alpha, \cX_\beta)>1$ for all $\beta<\alpha$.
Since $0<\alpha$, the preceding condition applies to $\beta=0$,
and as $\dist(0,\cX_0)=1$, it follows that $p_\alpha\forces\dot{x}_\alpha\neq0$.
Strengthening $p_\alpha$ if needed, we may assume that there is a countable set
$S\subseteq\omega_1$ such that
$$
p_\alpha\forces\supp(\dot{x}_\alpha)\subseteq S\ \wedge\ \sup\supp(\dot{x}_\alpha)=\sup S.
$$
Since $\alpha<\omega_1$, the family $\{\cX_\beta:\beta<\alpha\}$ is countable.
Applying Lemma~\ref{lem:IBmXn} to the condition $p_\alpha$, the set $S$, and an enumeration
of the family $\{\cX_\beta:\beta<\alpha\}$, we get a condition $q_\alpha\leq p_\alpha$ and
an $S$-center $\cY$ for $(q_\alpha,\dot{x}_\alpha)$ in $\ell_1(\omega_1)$ such that
$$
\dist(\cY,\cX_\beta)>1\quad\text{for all }\beta<\alpha
$$
and $\diam(\cY)<1$.
Pick $\alpha'\in\omega_1$ outside
$S\cup\{0\}\cup\bigcup\{S_{\cX_\beta}\cup\{\alpha_{\cX_\beta}\}:\beta<\alpha,\ \cX_\beta\neq \cX_0\}$.
For the translated center $\cX_\alpha:=\cY+1_{\alpha'}$, we have $\diam(\cX_\alpha)<1$,
$$
q_\alpha\forces \dist(\dot{x}_\alpha, \cX_\alpha)\leq 1
$$
and $\dist(\cX_\alpha, \cX_\beta)>2$ for all $\beta<\alpha$.

Define $\tilde{\X}:=\{\cX_\alpha:\alpha<\omega_1\}$.
Let $p\in\SSS^{\omega}$ and $\dot{x}$ be an $\SSS^{\omega}$-name for an
element of $\ell_1(\omega_1)$.
By the property of the family $\mathcal{N}$ we may assume that $\dot{x}\in\mathcal{N}$.
There is $0<\alpha<\omega_1$ such that
$(p,\dot{x})=(p_\alpha,\dot{x}_\alpha)$.
By the construction, there is $\beta\leq\alpha$ such that $q_\alpha\leq p$ and
$q_\alpha\forces\dist(\dot{x},\cX_\beta)\leq 1$.

\textbf{Case 2:} Assume that $\kappa=\omega_2$.
In this case, we first force with the auxiliary forcing $\mathbb{X}$.
Let $G_{\mathbb{X}}$ be an $\mathbb{X}$-generic filter over $V$.
By Lemma~\ref{lem:sigma}, the forcing $\mathbb{X}$ does not add new countable subsets of sets
from $V$, and thus $V[G_\mathbb{X}]$ satisfies $\mathsf{CH}$.
By Lemma~\ref{lem:kapcc}, the forcing $\mathbb{X}$ preserves cardinalities greater than or equal to
$\kappa$, in particular it preserves $\kappa$ and $\lambda$.
It follows that $\lambda^\omega=\lambda$, and the space $\ell_1(\kappa)$ and the forcing
$\SSS^{\lambda}$ do not change when we pass from $V$ to $V[G_\mathbb{X}]$.

Define $W:=V[G_\mathbb{X}]$ and $\tilde{\X}:=\bigcup G_\mathbb{X}$.
By the properties of conditions in $\mathbb{X}$ and the coding conventions fixed above, 
$\cX_0\in\tilde{\X}$, each $\cX\in\tilde{\X}\setminus\{\cX_0\}$ is a translated center,
and $\dist(\cX,\cX')>2$ for distinct $\cX,\cX'\in\tilde{\X}$.
Let $p\in\SSS^{\omega}$ and $\dot{x}$ be an $\SSS^{\omega}$-name for an element of
$\ell_1(\kappa)$. By continuous reading of names and the fact that
$\mathbb{X}$ adds no new countable sets, there are $p'\leq p$ and an
$\SSS^\omega$-name $\dot{y}\in V$ such that
$p'\forces\dot{x}=\dot{y}$.
By Lemma~\ref{lem:dense}, the set $\mathcal{D}_{p',\dot{y}}$ is dense in
$\mathbb{X}$. Since $G_\mathbb{X}$ is $\mathbb{X}$-generic over $V$, there is
$\X\in\mathcal{D}_{p',\dot{y}}\cap G_\mathbb{X}$.
Consequently, there are $q\leq p'$ and $\cX\in\tilde{\X}$ such that
$q\forces\dist(\dot{x},\cX)\leq1$.
\end{proof}

Let $V$ be a model satisfying $\mathsf{GCH}$ in which $\lambda>\kappa$ has uncountable
cofinality, and let $W$ and $\tilde{\X}$ be given by Proposition~\ref{prop:tilingfamily}.
Let $G$ be an $\SSS^{\lambda}$-generic filter over $W$.
We claim that the family $\{\mathrm{B}(\cX,1):\cX\in\tilde{\X}\}$ is a normal disjoint
tiling of $\ell_1(\kappa)$ by star-shaped bodies in $W[G]$.

First we prove that this family covers $\ell_1(\kappa)$:
For a countably infinite set $C$, let $i_C\colon \SSS^C\to\SSS^{\omega}$ be an isomorphism.
If $\dot{x}$ is an $\SSS^C$-name, then $i_C(\dot{x})$ denotes the corresponding
$\SSS^{\omega}$-name induced by this isomorphism.
Let $p\in\SSS^{\lambda}$ and $\dot{x}$ be an $\SSS^{\lambda}$-name for an element 
of $\ell_1(\kappa)$.
By continuous reading of names, strengthening $p$ if needed, we may assume 
that there is a countably infinite set $C\subseteq\lambda$ such that $p\in\SSS^C$
and $\dot{x}$ is an $\SSS^C$-name.
Using the property proved above for $\tilde{\X}$, there is $\cX\in\tilde{\X}$
and a condition $q\in\SSS^{\omega}$ such that 
$$
q\leq i_C(p)\quad\text{and}\quad q\forces \dist(i_C(\dot{x}),\cX)\leq 1.
$$
Equivalently, 
$$
i_C^{-1}(q)\leq p\quad\text{and}\quad i_C^{-1}(q)\forces \dist(\dot{x},\cX)\leq 1.
$$
Thus, the family $\{\mathrm{B}(\cX,1):\cX\in\tilde{\X}\}$ covers $\ell_1(\kappa)$ in
$W[G]$.

If $\cX,\cX'\in\tilde{\X}$ are distinct, then $\dist(\cX,\cX')>2$, and thus
$\mathrm{B}(\cX,1)\cap \mathrm{B}(\cX',1)=\emptyset$.
Since each $\cX\in\tilde{\X}$ is compact and $\diam(\cX)<1$, Lemma~\ref{B} implies
that $\mathrm{B}(\cX,1)=\cX+B(\ell_1(\kappa))$ is a star-shaped body.
Moreover, every such tile contains a ball of radius $1$ and has diameter at most $3$,
so the tiling is normal.

Finally, by the cardinal-preservation and continuum calculation for side-by-side
Sacks forcing from \cite{failure}, since $W$ satisfies $\mathsf{CH}$ and $\lambda^\omega=\lambda$, forcing with
$\SSS^{\lambda}$ over $W$ yields $\kappa<\lambda=2^\omega$.
This completes the proof. \hfill $\square$

\section{Negative results for countable cofinality}\label{sec:cofw}

In this section we work in {\sf ZFC} with $\ell_1(\kappa)$, where $\kappa$ has countable cofinality.

\begin{proposition}\label{cfomega}
Suppose that $\kappa$ is an uncountable cardinal
of countable cofinality. Then $\ell_1(\kappa)$ does not admit a disjoint  tiling by 
$\kappa$ many sets of the form $\mathrm{B}(X, r)$, where $X$ is  nonempty and norm-separable 
and $r>0$.
\end{proposition}
\begin{proof}
Let $(\kappa_n)_{n\in\N}$ be an increasing sequence of cardinal numbers below $\kappa$ which
has supremum $\kappa$.
Suppose that $\{\mathrm{B}(X_\xi,r_\xi):\xi<\kappa\}$ is such a disjoint tiling.
Fix $\gamma_0<\kappa$ and $x\in X_{\gamma_0}$.
Choose $\alpha\in\kappa\setminus\supp(X_{\gamma_0})$ and put
$z=x+r_{\gamma_0}\cdot1_\alpha$.
Then $\dist(z,X_{\gamma_0})=r_{\gamma_0}$, and so 
$z\in\mathrm{B}(X_{\gamma_0},r_{\gamma_0})$ and
$z\not\in \mathrm{B}(X_{\xi},r_{\xi})$ for any $\xi\in\kappa\setminus\{\gamma_0\}$.

Replacing each $X_\xi$ by $X_\xi-z$, we may therefore assume that
$$
\dist(0,X_{\gamma_0})=r_{\gamma_0}
\quad\hbox{and}\quad
\dist(0,X_\xi)>r_\xi\quad\hbox{for every }\xi\in\kappa\setminus\{\gamma_0\}.
$$
This replacement preserves the assumptions of the proposition, since $\supp(z)$ is countable.
Define
$$
A_n:=\bigcup\{\supp(X_\gamma):\gamma< \kappa_n\}
$$
for all $n\in\N$.
By induction on $k\in \N$ we construct  $n_k\in \N$ and $y_k\in
\ell_1(A_{n_k})$ with the following properties:
\begin{enumerate}
\item $n_k<n_{k+1}$,
\item $\dist(y_k, X_\gamma)=r_\gamma$ for some $\gamma<\kappa_{n_k}$,
\item $y_{l}|A_{n_k}=y_k$ for $l\geq k$,
\item $0<\|y_{k+1}|(A_{n_{k+1}}\setminus A_{n_k})\|\leq 1/2^k$,
\item if $y|A_{n_k}=y_k$ and $y|(\kappa\setminus A_{n_k})\not=0$, then
$y\not\in\bigcup\{\mathrm{B}(X_\gamma,r_\gamma):\gamma<\kappa_{n_k}\}$.
\end{enumerate}
Fix $n_0\in\N$ such that $\gamma_0<\kappa_{n_0}$ and put $y_0=0$.
Then item~(2) follows from the above translation of tiles, and item~(5) follows from
the additivity of the $\ell_1$ norm on disjoint supports.

Now let $k\in\N$ and suppose that $n_k$ and $y_k\in\ell_1(A_{n_k})$ have been
constructed and satisfy the relevant items~(1)--(5).
By~(2), fix $\gamma<\kappa_{n_k}$ such that
$\dist(y_k,X_\gamma)=r_\gamma$.
Let $t:=\frac{1}{2^k}$. 
Fix $\alpha\in\kappa\setminus A_{n_k}$, which is possible by the separability
of the $X_\gamma$s, by Lemma \ref{support} and since $\kappa_{n_k}<\kappa$.
By~(5) and since the tiling covers $\ell_1(\kappa)$, there is $\eta\geq \kappa_{n_k}$ such that
$$
y_k+t\cdot1_\alpha\in \mathrm{B}(X_\eta,r_\eta).
$$
Since $\gamma\neq\eta$, we have $y_k\notin\mathrm{B}(X_\eta,r_\eta)$, and thus $\dist(y_k,
X_\eta)>r_\eta$.
If $\alpha\notin\supp(X_\eta)$, then
$$
\dist(y_k+t\cdot1_\alpha,X_\eta)=\dist(y_k,X_\eta)+t>r_\eta,
$$
which is impossible.
Thus $\alpha\in\supp(X_\eta)$.
Since $\mathrm{B}(X_\eta,r_\eta)$ is closed, $y_k\notin\mathrm{B}(X_\eta,r_\eta)$, and
$y_k+t\cdot1_\alpha\in\mathrm{B}(X_\eta,r_\eta)$, there is $t'$ such that $0<t'\leq t$ and
$$
\dist(y_k+t'\cdot1_\alpha,X_\eta)=r_\eta.
$$
Fix $n_{k+1}>n_k$ such that $\eta<\kappa_{n_{k+1}}$.
Then $\alpha\in A_{n_{k+1}}\setminus A_{n_k}$.
Let $y_{k+1}:=y_k+t'\cdot 1_\alpha$.
Since the tiling is disjoint and $y_{k+1}\in\mathrm{B}(X_\eta,r_\eta)$, we have
$\dist(y_{k+1},X_\xi)\geq r_\xi$ for all $\xi<\kappa_{n_{k+1}}$.
Then items (1)--(4) follow from the construction and (5) follows from the additivity
of the $\ell_1$ norm on disjoint supports.

Finally, let $y\in \R^\kappa$ be given by $y|A_{n_k}=y_k$ for all $k\in\N$
and $y|(\kappa\setminus\bigcup_{k\in\N}A_{n_k})=0$.
Then $y$ is a well defined element of $ \ell_1(\kappa)$ by~(3) and~(4). 
Item~(5) implies that $y\not\in\mathrm{B}(X_\gamma,r_\gamma)$ for all $\gamma<\kappa$, a
contradiction.
\end{proof}

\section{Appendix: Cardinalities of Chebyshev sets in $\ell_1(\kappa)$}\label{sec:chebyshev}

If $D\subseteq \ell_1(\kappa)$  then 
the  corresponding \emph{Voronoi cells} are sets 
$$V_D(z):=\{x\in\ell_1(\kappa):\|x-z\|\leq\|x-w\|\text{ for every }w\in D\}$$
for $z\in D$. The set $V_D(z)$ is closed, since
$$
V_D(z)=\bigcap_{w\in D}\{x\in\ell_1(\kappa):\|x-z\|\leq\|x-w\|\}.
$$

\begin{lemma}\label{lines} Suppose that $\kappa>0$ is a cardinal.
Let $D\subseteq\ell_1(\kappa)$ be nonempty, $z\in D$, $\alpha\in\kappa$.
Then $V_D(z)\cap (a+\R1_{\alpha})$
 is either empty or a closed interval in $a+\R1_{\alpha}$ for every $a\in \ell_1(\kappa)$
satisfying $a(\alpha)=0$.
\end{lemma}

\begin{proof}
 Let $x,y\in V_D(z)\cap (a+\R1_{\alpha})$, and suppose that
$x(\alpha)\leq t\leq y(\alpha)$. Define  $u=a+t1_{\alpha}\in a+\R1_{\alpha}$.
We will show that $u\in V_D(z)\cap (a+\R1_{\alpha})$.
Fix $w\in D$, using the fact that $x|(\kappa\setminus\{\alpha\})=
y|(\kappa\setminus\{\alpha\})=u|(\kappa\setminus\{\alpha\})$ we have
$$
A:=\sum_{\beta\neq\alpha}|x(\beta)-z(\beta)|=\sum_{\beta\neq\alpha}|y(\beta)-z(\beta)|=
\sum_{\beta\neq\alpha}|u(\beta)-z(\beta)|.
$$
$$B:=\sum_{\beta\neq\alpha}|x(\beta)-w(\beta)|=\sum_{\beta\neq\alpha}|y(\beta)-w(\beta)|
=\sum_{\beta\neq\alpha}|u(\beta)-w(\beta)|.
$$
Since $x,y\in V_D(z)$, by the definition of the norm in $\ell_1(\kappa)$, we have
$$
|x(\alpha)-z(\alpha)|+A\leq |x(\alpha)-w(\alpha)|+B
\leqno (1)$$
and
$$
|y(\alpha)-z(\alpha)|+A\leq |y(\alpha)-w(\alpha)|+B.
\leqno(2)$$
Consider the continuous function $f:\R\rightarrow \R$ given  by $f(s)=
 |s-z(\alpha)|-|s-w(\alpha)|$. By (1) and (2) we have $f(x(\alpha)), f(y(\alpha))\leq B-A$.
The function $f$ is  monotone on $\R$. Indeed, on the interval with endpoints
$z(\alpha), w(\alpha)$ $f$ is nondecreasing if $z(\alpha)<w(\alpha)$ and
nonincreasing if $z(\alpha)>w(\alpha)$ and it is constant 
on the halflines outside this interval  and on the entire $\R$ if $z(\alpha)=w(\alpha)$. Hence, since
$t\in[x(\alpha),y(\alpha)]$, we have $f(t)\leq B-A$ so
$$
|u(\alpha)-z(\alpha)|+A\leq |u(\alpha)-w(\alpha)|+B.
$$
Again by  the definition of the norm in $\ell_1(\kappa)$
we conclude that
$$
\|u-z\|\leq\|u-w\|.
$$
As $w\in D$ was arbitrary, $u\in V_D(z)$. Thus $V_D(z)\cap (a+\R1_{\alpha})$ is a
closed interval as needed.
\end{proof}

\begin{proposition}\label{prop:chebyshev}
Suppose that $\kappa$ is a cardinal. Every Chebyshev set in
$\ell_1(\kappa)$ containing at least two points has cardinality at least
$2^\omega$.
\end{proposition}

\begin{proof} 
Let $D\subseteq\ell_1(\kappa)$ be a Chebyshev set containing distinct points
$z,w$. Then
$$
\{V_D(v):v\in D\}
$$
is a pairwise disjoint cover of $\ell_1(\kappa)$ by closed sets.

We first show that there is a line $L_{a, \alpha}=a+\R1_{\alpha}$
where $a(\alpha)=0$ and $a\in \ell_1(\kappa)$ which meets at least two distinct
Voronoi cells $V_D(x)$ for $x\in D$. Suppose otherwise. 

The set $\supp(w-z)$ is nonempty and
countable, so $\supp(w-z)=\{\alpha_n:n\in\N\}$ for some sequence
$(\alpha_n)_{n\in\N}$, where repetitions are allowed.  Put
$$
x_n:=z+(w-z)|\{\alpha_0,\ldots,\alpha_{n-1}\}
$$
for $n\in\N$. Then $x_0=z\in V_D(z)$ and $x_n$ converges to $w$. 
So, since $w\not\in V_D(z)$ there is $k\in \N$ which is minimal such that $x_{k+1}\not\in V_D(z)$.
As $x_k$ and $x_{k+1}$ differ only at the coordinate $\alpha_k$
they are both in the line $L_{a_k, \alpha_k} $ for appropriate $a_k$ satisfying
 $a_k(\alpha_k)=0$ and $a_k\in \ell_1(\kappa)$. The minimality of $k$ implies
that $L=L_{a_k, \alpha_k}$ meets at least two distinct sets $V_D(x)$.

Identify $L$ with $\R$. By Lemma~\ref{lines}, the family
$\{V_D(v)\cap L:v\in D,\ V_D(v)\cap L\neq\emptyset\}$
is a nontrivial partition of $\R$ into closed intervals. By the Sierpi\'nski theorem
from \cite{sierpinski} it is uncountable.
There can be only countably many nondegenerate members of it,  so the set of points
which form singleton members together with the endpoints of the
nondegenerate members is an uncountable closed subset of $L$. Hence it has
cardinality $2^\omega$.
Consequently the partition has cardinality $2^\omega$
and so
$|D|\geq 2^\omega$.
\end{proof}

\begin{theorem}\label{thm:mainchebyshev} Suppose that $\kappa<2^\omega$ is infinite. Then
the Banach space $\ell_1(\kappa)$ does not admit a nontrivial discrete Chebyshev set.
\end{theorem}
\begin{proof}
Since the density of $\ell_1(\kappa)$ is $\kappa$, it cannot admit  discrete sets bigger than $\kappa$.
Now the theorem follows from  Proposition \ref{prop:chebyshev}.
\end{proof}

\section*{Role of a digital assistant}

A digital assistant was used at the final stage of the preparation 
of this paper for mathematical discussions of the initial version of the paper
(which included all mathematical ideas and results besides those in the Appendix), 
checking arguments and improving the exposition. 
At this stage the assistant made a direct mathematical contribution to what is now
 the Appendix (nonexistent in the initial human-prepared version): 
the analysis of intersections of Voronoi cells with coordinate lines, 
initiated by the authors, led in the discussion with the digital
 assistant to Lemma~\ref{lines} (now with a different, human-readable proof). 
Unsolicited, the digital assistant subsequently proposed and proved Proposition~\ref{prop:chebyshev} 
and observed that these results yield Theorem~\ref{thm:mainchebyshev}. 
The authors checked and revised all arguments 
and take responsibility for the final statements and proofs.

\bibliographystyle{amsplain}

\end{document}